\documentclass[10pt,reqno]{amsart}
\usepackage{amssymb, amscd,url}
\usepackage{color}   

\newtheorem{theorem}{Theorem}
\newtheorem{proposition}[theorem]{Proposition} 
\newtheorem{corollary}[theorem]{Corollary}
\newtheorem{lemma}[theorem]{Lemma}
\newtheorem{question}[theorem]{Question}

\theoremstyle{definition}

\newtheorem{remark}[theorem]{Remark}

\def\botimes{\mathbin{\bar{\otimes}}}

\begin{document}
\title[Free product pair rigidity]{A free product pair rigidity result\\ in von Neumann algebras}
\author[Y.~Ueda]{Yoshimichi Ueda}
\thanks{{\it Current affiliation and e-mail address}: (since 1st Oct.\ 2017) Graduate School of Mathematics, Nagoya University,
Furocho, Chikusaku, Nagoya, 464-8602, Japan; \url{ueda@math.nagoya-u.ac.jp}}
\address{
Graduate School of Mathematics, 
Kyushu University, 
Fukuoka, 810-8560, Japan
}
\email{ueda@math.kyushu-u.ac.jp}
\date{Sep. 3rd, 2017}
\begin{abstract} 
We prove that the free product pair of finitely many copies of the unique amenable type III$_1$ factor endowed with weakly mixing states remembers the number of free components and the given states.
\end{abstract}
\maketitle

\allowdisplaybreaks{

\section{Introduction} The free product construction for von Neumann algebras is a method of constructing a pair of von Neumann algebra and faithful normal state from given such pairs. Here we are interested in how much information about given pairs can be restored from the resulting pair in the construction. Moreover, we are particularly interested in the case where given von Neumann algebras are {\it amenable} and/or given faithful normal states are \emph{with small centralizer}. This was originally motivated from the following: Firstly, all the available Kurosh-type rigidity results need the non-amenability of given von Neumann algebras; secondly faithful normal states with small centralizer are usually less tractable than ones with large centralizer such as almost periodic states in the study of non-tracial free products (see e.g. \cite{Ueda:MRL11},\cite{Ueda:AmerJMath16}). The existing Kurosh-type rigidity results in von Neumann algebras are: Ozawa's pioneer work \cite{Ozawa:IMRN06} (for given weakly exact, non-prime non-amenable type II$_1$ factors); Ioana--Peterson--Popa's epoch-making work \cite{IoanaPetersonPopa:ActaMath08} (for given type II$_1$ factors possessing regular diffuse von Neumann subalgebras with relative property (T)); Peterson's striking work  \cite{Peterson:InventMath09} (for given non-amenable $L^2$-rigid type II$_1$ factors); Asher's work \cite{Asher:PAMS09} (generalizing Ozawa's to non-tracial states); the latest work \cite{HoudayerUeda:ComposMath16} of Houdayer {\it et al.} (unifying and generalizing the previous results to the setting of arbitrary states).  

\medskip
Let $R_\infty$ be the unique amenable type III$_1$ factor and $\{\varphi_i\}_{i=1}^m$ and $\{\psi_j\}_{j=1}^n$ be finite families of weakly mixing states on $R_\infty$. (See section 3 for the definition of weakly mixing states.) Recall that a von Neumann algebra with weakly mixing state must be either trivial or a type III$_1$ factor. (This is well known in the algebraic quantum field theory, see \cite[Corollary 1.0.8]{Baumgartel:Book}.) Consider two free product pairs $(M,\varphi) := \bigstar_{i=1}^m(R_\infty,\varphi_i)$ and $(N,\psi) := \bigstar_{j=1}^n(R_\infty,\psi_j)$. The resulting $M$ and $N$ are known to be type III$_1$ factors and the centralizers $M_\varphi$ and $N_\psi$ are also known to be trivial (\cite[Lemma 7]{Barnett:PAMS95}; see also the proof of \cite[Proposition 2.1]{Ueda:MRL11}). In what follows, $(M,\varphi) \cong (N,\psi)$ means that there exists a bijective $*$-homomorphism $\pi : M \to N$ such that $\psi = \pi_*(\varphi) := \varphi\circ\pi^{-1}$. We will explicitly write the canonical embedding maps $\lambda^M_i : R_\infty \to M$, $1 \leq i \leq m$, and $\lambda^N_j : R_\infty \to N$, $1 \leq j \leq n$, so that $\varphi\circ\lambda^M_i = \varphi_i$ and $\psi\circ\lambda^N_j = \psi_j$. The main theorem of this paper is the next new rigidity phenomenon (which we call {\it the free product pair rigidity}) for a certain class of type III$_1$ free product factors. It unexpectedly arose as a bonus of considering the case where given states are with `extremely small centralizer'. 
     
\begin{theorem}\label{T1} 
If there exists a bijective $*$-homomorphism $\pi : M \to N$ with $\psi = \pi_*(\varphi)$, then $m=n$ and there exists a unique permutation $\kappa = \kappa_\pi \in \mathfrak{S}_m$ such that for each $1 \leq i \leq m$, $\pi(\lambda^M_i(R_\infty)) = \lambda^N_{\kappa(i)}(R_\infty)$ holds and $\pi_i := (\lambda^N_{\kappa(i)})^{-1}\circ(\pi\!\upharpoonright_{\lambda^M_i(R_\infty)})\circ\lambda^M_i \in \mathrm{Aut}(R_\infty)$ satisfies $\psi_{\kappa(i)} = (\pi_i)_*(\varphi_i)$. Conversely, if $m=n$, then any $\kappa \in \mathfrak{S}_m$ and $\pi_i \in \mathrm{Aut}(R_\infty)$ with $(\pi_i)_*(\varphi_i)=\psi_{\kappa(i)}$, $1 \leq i \leq m$, give rise to a unique bijective $*$-homomorphism $\pi : M \to N$ such that $\pi\circ\lambda^M_i = \lambda^N_{\kappa(i)}\circ\pi_i$ for all $1 \leq i \leq m$ and moreover that $\pi_*(\varphi)=\psi$. 

In particular, $(M,\varphi) \cong (N,\psi)$ if and only if  $m=n$ and $(R_\infty,\varphi_i) \cong (R_\infty,\psi_i)$, $1 \leq i \leq m$, after permutation on the indices. 
\end{theorem}

The next corollary is immediate from the above theorem. 

\begin{corollary}\label{C2} If all the $\varphi_i$ are identical to a fixed weakly mixing state $\varphi_0$, then the $\varphi$-preserving automorphism group $\mathrm{Aut}_\varphi(M)$ is isomorphic to the wreath product $\mathfrak{S}_m \ltimes \mathrm{Aut}_{\varphi_0}(R_\infty)^m$ by $\pi \mapsto (\kappa_\pi,(\pi_i))$ with notation in Theorem \ref{T1}, where $\kappa \in \mathfrak{S}_m$ acts on $\mathrm{Aut}_{\varphi_0}(R_\infty)^m$ by $(\pi_i)^\kappa := (\pi_{\kappa(i)})$.
\end{corollary}

The same assertion of Theorem \ref{T1} still holds even when some copies of $R_\infty$ in the free products $(M,\varphi)$, $(N,\psi)$ are replaced with anti-freely indecomposable non-amenable type III$_1$ factors (in the sense of \cite{HoudayerUeda:ComposMath16}) endowed with weakly mixing states. For such a generalization, we only need to use the intermediate assertion ($\diamondsuit$) in the proof of \cite[Main Theorem]{HoudayerUeda:ComposMath16} instead of Lemma \ref{L7} (but Lemma \ref{L8} is necessary) when dealing with anti-freely indecomposable non-amenable type III$_1$ factors. 

\medskip
We briefly mention some crucial ideas appearing in the proof of Theorem \ref{T1}. Although Theorem \ref{T1} is a kind of Kurosh-type rigidity result, any existing Kurosh-type rigidity results cannot be used to prove Theorem \ref{T1}, since they need, among others, the non-amenability of given von Neumann algebras. Instead, we use a simple analysis of central sequences in the presence of weak mixing property (see Lemma \ref{L6}). This type of analysis, whose prototype appears in an old work of Popa \cite{Popa:AdvMath83} on type II$_1$ factors, was used by Houdayer--Raum \cite{HoudayerRaum:MathAnn15} and Boutonnet--Houdayer \cite{BoutonnetHoudayer:APDE16} in combining with the weak mixing property for some questions on free Araki--Woods factors. In this respect, the novelty of this work is the use of a bounded projection onto the span of `letters' (see equation \eqref{Eq1} and Lemma \ref{L5}) in such an analysis. Despite this difference between the existing Kurosh-type rigidity results and Theorem \ref{T1}, our essential strategy of proving Theorem \ref{T1} still follows the fundamental principle of \emph{Popa's deformation/rigidity theory} (see \cite{Popa:ICM06}). Namely, we will use Popa's intertwining technique under `tensions' between \emph{rigidity} and \emph{malleability}. In the proof of Theorem \ref{T1}, the required \emph{rigidity} comes from the weak mixing property of the modular actions of given states and the required \emph{malleability} does from the amenability of given von Neumann algebras.  Finally, we point out that the key observations, Lemmas \ref{L7} and \ref{L8}, are applicable to any free group factors of finite rank (see section 4). However, it is unclear whether or not our observations, especially Lemma \ref{L7}, will be able to give any contributions to the (non-)isomorphism problem. 

\medskip
Necessary backgrounds on free products with respect to arbitrary faithful normal states can be found in \cite[section 2]{Ueda:AdvMath11}, and the other necessary facts will be referred to suitable references at appropriate places. Unlike \cite[section 2]{Ueda:AdvMath11} we will use the standard form $(M,L^2(M),J^M,\mathfrak{P}^M)$ of a given von Neumann algebra $M$ (see e.g. \cite[Chapter IX, section 1]{Takesaki:Book2}) instead of a GNS representation. It is well known that any positive $\varphi \in M_*$ has a unique implementing vector $\xi_\varphi \in \mathfrak{P}^M$, i.e., $(x\xi_\varphi\,|\,\xi_\varphi)_{L^2(M)} = \varphi(x)$ for $x \in M$ and moreover that $\xi_\varphi$ becomes cyclic and separating if and only if $\varphi$ is faithful. We also use the notation concerning ultraproduct von Neumann algebras such as $(x_n)^\omega$ in \cite{AndoHaagerup:JFA14}, which is a bit different from that in \cite[section 2]{Ueda:AdvMath11}. In what follows, we say that a von Neumann subalgebra is {\it with expectation}, if there exists a faithful normal conditional expectation from the ambient von Neumann algebra onto it.  

\section{Yet another variant of Popa's criterion} 

In \cite[Appendix]{Popa:AnnMath06},\cite[section 2]{Popa:InventMath06} Popa invented quite a powerful criterion for deciding whether or not a von Neumann subalgebra $A$ of a finite von Neumann algebra $M$ can be conjugated into another $B$ of $M$ by a partial isometry in $M$. His criterion is the most key ingredient of the intertwining technique, which has played an important r\^{o}le in many works since then. To prove Theorem \ref{T1}, we also need to use Popa's criterion in a situation where $M$ is a type III$_1$ factor and $A, B$ are amenable type III$_1$ subfactors with expectation. However, the previously known `non-tracial' variants of Popa's criterion cannot be used in such a situation. Hence we prepare yet another variant of Popa's criterion for proving Theorem \ref{T1}.

\medskip
Let $M$ be a $\sigma$-finite von Neumann algebra and $A, B \subseteq M$ be two (not necessarily unital) von Neumann subalgebras with a faithful normal conditional expectation $E_B : 1_B M 1_B \to B$, where $1_B$ denotes the unit of $B$. The proposition below was observed in the fall of 2014 and triggered by a discussion with Yusuke Isono about an unpublished attempt due to Deprez and Raum, which also triggered Houdayer--Isono's thoroughgoing work \cite[section 4]{HoudayerIsono:AdvMath17} on Popa's criterion in the not necessarily finite von Neumann algebra setting. Although Houdayer--Isono's variant of Popa's criterion seems more useful than our variant in general, it requires $A$ to be finite and hence cannot be used to prove Theorem \ref{T1}. Our main technical contribution is only the use of amenability in place of the so-called minimal distance theorem in Hilbert spaces, but we do give the complete proof of our variant, since it is certainly necessary for proving Theorem \ref{T1}. 

\begin{proposition}\label{P3} Assume that $A$ is amenable. Then the following are equivalent{\rm:}
\begin{itemize}
\item[(1)] There exists no net $u_i$ of unitaries in $A$ satisfies that $\lim_i E_B(y^* u_i x) = 0$ strongly for all $x,y \in 1_A M 1_B$.
\item[(2)] There exist a natural number $n$, a non-zero partial isometry $v \in \mathbb{M}_{n,1}(M)$ and a normal {\rm(}possibly non-unital{\rm)}  $*$-homomorphism $\theta : A \to \mathbb{M}_n(B)$ such that $va = \theta(a)v$ for every $a \in A$. 
\end{itemize}
\end{proposition}

Following Popa's notation, we write $A \preceq_M B$ and say that $A$ embeds into $B$ inside $M$, if the above equivalent conditions hold. 

\begin{proof} 
It suffices to show that item (1) implies item (2), for the opposite direction is shown in the usual way (see \cite[Proposition C.1 (1) $\Rightarrow$ (4)]{Vaes:Asterisque07}). 

Let $z$ be the maximal central projection in $B$ so that $Bz$ is finite; hence $B(1_B-z)$ is properly infinite. Choose and fix a faithful normal state $\varphi$ on $B$ in such a way that $\varphi\!\upharpoonright_{Bz}$ is tracial. Remark that item (1) is equivalent to that there exist a finite subset $\mathcal{F} \Subset 1_A M1_B$ and $\varepsilon > 0$ so that $\sum_{x,y \in \mathcal{F}} \Vert E_B(y^* u x)\xi_\varphi\Vert_{L^2(B)}^2 \geq \varepsilon$ for all unitaries $u \in A$.

Consider the unitalization $B^\sim := B + \mathbb{C}1_B^\perp$ inside $M$, and we can extend $E_B$ and $\varphi$ to a faithful normal conditional expectation $(E_B)^\sim : M \to B^\sim$ and a faithful normal positive linear functional $\varphi^\sim$ on $B^\sim$, respectively. 

Applying the well-known representation theory of von Neumann algebras to the restriction of the right action of $M$ on $L^2(M)$ to $B^\sim$, we can identify the basic extension $\langle M,B^\sim\rangle$ (by $(E_B)^\sim$) with a certain reduced algebra $P(B(\ell^2)\botimes B^\sim)P$ with projection $P \in B(\ell^2)\botimes B^\sim$ in such a way that $e_{11}\otimes 1 \leq P$ and that the identification sends ${}^t r := J^M r J^M$ with projection $r \in \mathcal{Z}(B^\sim)$, the Jones projection $e_{B^\sim}$ and $b^\sim e_{B^\sim}$ with $b^\sim \in B^\sim$ to $P(1_{B(\ell^2)}\otimes r) = (1_{B(\ell^2)}\otimes r)P$, $e_{11}\otimes 1$ and $e_{11}\otimes b^\sim$, respectively, where $e_{ij}$ are canonical matrix units in $B(\ell^2)$. See the proof of \cite[Proposition 3.1 (i) $\Rightarrow$ (ii)]{Ueda:JLMS13}. In particular, we have ${}^t r e_{B^\sim} = re_{B^\sim}$ ($= e_{11}\otimes r$ in $P(B(\ell^2)\botimes B^\sim)P$). Hence we can construct a faithful normal semifinite trace $\mathrm{Tr}$ on ${}^t z\,\langle M,B^\sim\rangle$ by transferring $\mathrm{Tr}_{B(\ell_2)}\botimes(\varphi\!\upharpoonright_{Bz})$ on $B(\ell_2)\botimes Bz$ to ${}^t z\,\langle M,B^\sim\rangle$. Remark that $\mathrm{Tr}$ is characterized by $\mathrm{Tr}(b e_{B^\sim}) = \varphi(b)$ for $b \in Bz$, because $be_{B^\sim} = bze_{B^\sim} = b\,{}^t z\,e_{B^\sim}$ in ${}^t z\,\langle M,B^\sim\rangle$ and this element corresponds to $e_{11}\otimes b$ in $B(\ell^2)\botimes Bz$. 

Let $\widehat{(E_B)^\sim} : \langle M,B \rangle \to M$ be the dual operator valued weight (see e.g. \cite[subsection 2.1]{IzumiLongoPopa:JFA98}). Set $d := \sum_{y \in \mathcal{F}} y e_{B^\sim} y^* \in 1_A\,{}^t 1_B\,\langle M, B^\sim\rangle1_A\,{}^t 1_B$, where $1_A$ denotes the unit of $A$. Then $\widehat{(E_B)^\sim}(d) = \sum_{y \in \mathcal{F}} yy^*$. We have
\begin{align*}
\mathrm{Tr}({}^t z d) 
= 
\sum_{y \in \mathcal{F}} \mathrm{Tr}(({}^t z\,y e_{B^\sim})({}^t z\,y e_{B^\sim})^*) 
= 
\sum_{y \in \mathcal{F}} \mathrm{Tr}({}^t z\, e_{B^\sim} y^* y e_{B^\sim}) 
= \sum_{y \in \mathcal{F}}\varphi(zE_B(y^* y)) < +\infty. 
\end{align*}

Set $\mathcal{C} := \overline{\mathrm{co}}^\text{$\sigma$-w}\{ u^* d u \mid \text{unitary $u \in A$}\}$ inside $1_A\,{}^t 1_B\,\langle M, B^\sim\rangle1_A\,{}^t 1_B$. \emph{Since $\mathcal{C}$ is a $\sigma$-weakly compact convex subset and since $A$ is amenable, there exists a fixed point $c_0 \in \mathcal{C}$ under the adjoint action of the unitary group of $A$.} This is immediate when the predual $A_*$ is separable; see e.g.\ the proof of \cite[Theorem 3.9]{IzumiLongoPopa:JFA98}. The general case where $A_*$ is not separable needs \cite[Theorem 4]{Elliott:CanadMathBull78} in addition. In fact, it implies that $A$ is generated by an upward directed collection of hyperfinite (in the classical sense, see \cite[\S1]{Elliott:CanadMathBull78}) von Neumann subalgebras $A_j$, $j \in J$, to each of which the proof of \cite[Theorem 3.9]{IzumiLongoPopa:JFA98} is applicable as above. Namely, the subset $\mathcal{C}_j$ of fixed-points in $\mathcal{C}$ under the adjoint action of the unitary group of each $A_j$ is not empty. Since the collection $A_j$, $j \in J$, is upward directed, it is easy to confirm that the intersection of any finite sub-collection of the collection $\mathcal{C}_j$, $j \in J$, contains some $\mathcal{C}_j$ and thus is not empty. Hence the intersection of all the $\mathcal{C}_j$, $j\in J$, is a non-empty set, and a desired element $c_0$ lives there. 

Using the equivalent condition of item (1) given in the second paragraph, we have 
$$
\sum_{x \in \mathcal{F}}((u^* d u) x\xi_{\varphi^\sim\circ(E_B)^\sim}\,|\,x\xi_{\varphi^\sim\circ(E_B)^\sim})_{L^2(M)} = \sum_{x,y\in\mathcal{F}}\Vert E_B(y^* u x)\xi_\varphi\Vert_{L^2(B)}^2 \geq \varepsilon
$$ 
for every unitary $u \in A$, and hence $\sum_{x \in \mathcal{F}}(c x\xi_{\varphi^\sim\circ(E_B)^\sim}\,|\,x\xi_{\varphi^\sim\circ (E_B)^\sim})_{L^2(M)} \geq \varepsilon$ for all $c \in \mathcal{C}$. In particular, $c_0$ is non-zero. Since $\widehat{(E_B)^\sim}(d)\in M$, we have $\chi\circ\widehat{(E_B)^\sim}(c_0) \leq \Vert \widehat{(E_B)^\sim}(d)\Vert_\infty \Vert \chi\Vert$ for every positive $\chi \in M_*$ by lower semicontinuity, and thus it is plain to see, by using its generalized spectral decomposition (see \cite[Theorem 1.5]{Haagerup:JFA79-1}), that $\widehat{(E_B)^\sim}(c_0)$ falls in $M$. Similarly, by lower semicontinuity we have $\mathrm{Tr}({}^t z c_0) \leq \mathrm{Tr}({}^t z d) < +\infty$. Taking a suitable spectral projection of $c_0$ we can find a non-zero projection $f \in A' \cap 1_A\,{}^t 1_B\,\langle M, B^\sim\rangle1_A\,{}^t 1_B$ in such a way that $\widehat{(E_B)^\sim}(f) \in M$ and $\mathrm{Tr}(p) < \infty$ with $p:={}^t z f \in {}^t z \langle M,B^\sim\rangle$. 

\emph{With $\langle M,B^\sim\rangle = P(B(\ell^2)\botimes B^\sim)P$, the projection ${}^t(1_B-z)\,e_B$ is nothing but $e_{11}\otimes(1_B-z)$ being properly infinite, since $B(1_B-z)$ is properly infinite; hence $q := {}^t(1_B-z)\,f = f-p$ is subequivalent to ${}^t(1_B-z)\,e_{B^\sim}$ inside $\langle M,B^\sim\rangle$ by \cite[Theorem 6.3.4]{KadisonRingrose:Book2}.} By the same reason as in the beginning of the proof of \cite[Proposition F.10]{BrownOzawa:Book} or by \cite[Lemma A.1]{Vaes:Asterisque07} we may and do assume, with cutting $p$ by a central projection if necessary, that $p$ is subequivalent to $\mathrm{diag}({}^t z\,e_{B^\sim},\dots,{}^t z\,e_{B^\sim})$ inside the $n$-amplification $\mathbb{M}_n({}^t z \langle M,B^\sim\rangle)$ for some finite $n$. Consequently, there exists a partial isometry $V \in \mathbb{M}_n(\langle M,B^\sim\rangle)$ so that $V^* V = \mathrm{diag}(f,0,\dots,0)$ and $VV^* \leq e_B^{(n)} := \mathrm{diag}(e_B,\dots,e_B)$ with $e_B := {}^t 1_B e_{B^\sim} = 1_B e_{B^\sim}$ as before. Define a normal $*$-homomorphism $\theta : A \to \mathbb{M}_n(B)$ is defined by 
$$
a \in A \mapsto V\mathrm{diag}(af,0,\dots,0)V^* \in  e_B^{(n)}\mathbb{M}_n(\langle M,B^\sim\rangle)e_B^{(n)} = \mathbb{M}_n(B)e_B^{(n)} \cong \mathbb{M}_n(B).
$$
It follows that $V\mathrm{diag}(a,0,\dots,0) = \theta(a)V$ for every $a \in A$. Apply $\widehat{(E_B)^\sim}\otimes\mathrm{Id}_n$ to this equation ({\it n.b.}, we can do since $\widehat{(E_B)^\sim}(f) \in M$), and the push-down lemma \cite[Proposition 2.2]{IzumiLongoPopa:JFA98} (which clearly holds without the factoriality assumption on given algebras) with the help of (the proof of) \cite[Lemma 4.5]{Haagerup:JFA79-1} shows that there exists $w \in \mathbb{M}_n(M)$ such that $w^*w \leq \mathrm{diag}(1,0,\dots,0)$ and $w\,\mathrm{diag}(a,0,\dots,0) = \theta(a)\,w$ for every $a \in A$. Taking the polar decomposition of $w$ we get a desired element $v$.      
\end{proof}

\begin{remark}\label{R4} Proposition \ref{P3} can be a bit strengthened, as in Houdayer and Isono's work \cite{HoudayerIsono:AdvMath17}, in the  following way: If it is further assumed that there exists a faithful normal conditional expectation $E_A : 1_A M 1_A \to A$, then we can make the inclusion $\theta(A) \subseteq \theta(1_A)\mathbb{M}_n(B)\theta(1_A)$ with expectation.
\end{remark} 
\begin{proof} 
Let $f$ be as in the above proof; namely, $f$ is a non-zero projection in $1_A{}^t 1_B\,\langle M,B^\sim\rangle 1_A {}^t 1_B$ such that $\mathrm{diag}(f,0\dots0) \precsim e_B^{(n)}$ inside $\mathbb{M}_n(\langle M,B^\sim\rangle)$ and $\widehat{(E_B)^\sim}(f) \in M$. Set $\Phi := (E_A)^\sim\circ\widehat{(E_B)^\sim}$, a faithful normal operator valued weight from $\langle M,B^\sim\rangle$ onto $A^\sim$, where $A^\sim$ is the unitalization of $A$ inside $M$ and $(E_A)^\sim : M \to A^\sim$ is  a faithful normal conditional expectation extending $E_A$. Observe $\Phi(f) \in \mathcal{Z}(A)$, and thus one can choose a non-zero spectral projection $e \in \mathcal{Z}(A)$ of $\Phi(f)$ with $\Phi(f)e \geq \delta e$ for some $\delta > 0$. Observe that $ef$ is still a non-zero projection in $A' \cap 1_A {}^t 1_B\,\langle M,B^\sim\rangle\,1_A {}^t 1_B$ (since $e \in \mathcal{Z}(A^\sim)$). Replacing $f$ with $ef$ we may and do assume that $\Phi(f) \geq \delta s$ and $f \leq s$ with denoting by $s$ the support projection of $\Phi(f)$. We can choose a non-zero positive element $c \in \mathcal{Z}(A)$ so that $\Phi(f)\,c = c\,\Phi(f) = s$. The mapping $x \in f\langle M,B^\sim\rangle f \mapsto \Phi(x)cf \in Af$ gives a faithful normal conditional expectation. 
Therefore, the desired assertion follows, since $Af \subseteq f\langle M,B^\sim\rangle f$ is conjugate to $\theta(A) \subseteq \theta(1_A)\mathbb{M}_n(B^\sim)\theta(1_A)$ by the composition of the mapping 
$$
x \in f\langle M,B^\sim\rangle f \mapsto \mathrm{diag}(x,0\dots,0) \mapsto V\mathrm{diag}(x,0\dots,0)V^* \in e_B^{(n)}\mathbb{M}_n(\langle M,B^\sim\rangle)e_B^{(n)}
$$ 
and the inverse of $y \in \mathbb{M}_n(B) \mapsto ye_B^{(n)} \in \mathbb{M}_n(B)e_B^{(n)} = e_B^{(n)}\mathbb{M}_n(\langle M,B^\sim\rangle)e_B^{(n)}$ in this order.  
\end{proof}

\section{Proof of Theorem \ref{T1}} 

We begin by recalling the weak mixing property for group actions as well as for faithful normal states. Let $M$ be a von Neumann algebra with a distinguished faithful normal state $\varphi$. A $\varphi$-preserving action $\alpha : G \curvearrowright N$ of a second countable, locally compact group $G$ is said to be weakly mixing if its canonical implementing unitary representation of $\pi: G \curvearrowright L^2(M)$ (defined by $\pi(g) x\xi_\varphi := \alpha_g(x)\xi_\varphi$, $x \in M$) satisfies that for every finite subset $\Omega$ of a given dense subset of $L^2(M)\ominus\mathbb{C}\xi_\varphi$ and every $\varepsilon > 0$ there exists a $g \in G$ such that $|(\pi(g)\xi|\zeta)_{L^2(M)}| < \varepsilon$ for all $\xi,\zeta \in \Omega$.  We also say that the state $\varphi$ itself is weakly mixing if its modular automorphism group $\sigma^\varphi : \mathbb{R} \curvearrowright M$ is weakly mixing. 

\medskip
Let $(M,\varphi)=\bigstar_{i=1}^m(M_i,\varphi_i)$ be a non-trivial free product of $\sigma$-finite von Neumann algebras endowed with faithful normal states. Let $E_{M_i} : M \to M_i$ be the unique $\varphi$-preserving conditional expectation. We will regard each $M_i$ as a von Neumann subalgebra of the resulting free product von Neumann algebra $M$ and hence omit the canonical embedding map from $M_i$ into $M$. 

\medskip
We first introduce a normal bounded projection onto the subspace spanned by `letters' as follows. Set 
\begin{equation}\label{Eq1}
\Phi(x) : = \varphi(x)1 + \sum_{i=1}^m E_{M_i}(x - \varphi(x)1), \quad x \in M.
\end{equation}
Clearly, this is a normal self-adjoint linear map from $M$ to itself. Moreover, we have the following properties: 

\begin{lemma}\label{L5} The following hold true: 
\begin{itemize}
\item[(1)] $\Phi\circ\Phi = \Phi$ and $\Phi(M) = M_1 + \cdots + M_m$ . 
\item[(2)] $\Phi\circ\sigma_t^\varphi = \sigma_t^\varphi\circ\Phi$ for every $t \in \mathbb{R}$. 
\item[(3)] $\Phi^\omega : (x_n)^\omega \in M^\omega \mapsto (\Phi(x_n))^\omega \in M^\omega$ defines a well-defined normal self-adjoint idempotent map and its range is exactly $M_1^\omega + \cdots + M_m^\omega$. 
\end{itemize}
\end{lemma}
\begin{proof} 
(1) is just a computation by using $E_{M_i}(\mathrm{Ker}(\varphi)) = \mathrm{Ker}(\varphi_i)$ and $E_{M_j}(\mathrm{Ker}(\varphi_i)) = \{0\}$ as long as $i\neq j$.

(2) follows from the well-known fact that $\varphi\circ\sigma_t^\varphi = \varphi$ and $E_{M_i}\circ\sigma_t^\varphi = \sigma_t^\varphi\circ E_{M_i}$ for $t \in \mathbb{R}$. 

(3) follows from that the normal self-adjoint map
$$
\Psi : x \in M^\omega \mapsto 
\varphi^\omega(x)1 + \sum_{i=1}^m E_{M_i^\omega}(x-\varphi^\omega(x)1) \in M^\omega
$$
satisfies that $\Psi((x_n)^\omega) = (\Phi(x_n))^\omega = \Phi^\omega((x_n)^\omega)$ for every $(x_n)^\omega \in M^\omega$ by definition. Here $E_{M_i^\omega} : M^\omega \to M_i^\omega$ denotes the $\varphi^\omega$-preserving conditional expectation. 
\end{proof}

We will provide three {\it general} lemmas (Lemmas \ref{L6}--\ref{L8}), where we always assume that for every $1 \leq i \leq m$ there exists a $\varphi_i$-preserving weakly mixing action $\alpha^{(i)} : G \curvearrowright M_i$ of a common second countable, locally compact group $G$ and let $\alpha = \bigstar_{i=1}^m \alpha^{(i)} : G \curvearrowright M$ be the so-called free product action, i.e., $\alpha\!\upharpoonright_{M_i} = \alpha^{(i)}$ for each $1 \leq i \leq n$. This action gives the required rigidity for making the fundamental principle of Popa's deformation/rigidity theory work well. 

\medskip
The next lemma will be proved by combining the ideas of \cite[Theorem 3.1]{Houdayer:CMP15} (whose original idea dates back to \cite[subsection 2.1]{Popa:AdvMath83}; also see \cite[Theorem 3.1]{HoudayerUeda:MPCPS16} for its finalized version) and \cite[Theorem 4.3]{HoudayerRaum:MathAnn15} (also see \cite{BoutonnetHoudayer:APDE16}). The novelty of the next lemma is the use of the normal self-adjoint linear map $\Phi : M \to M$ introduced above in place of $E_{M_i}$. The key observation behind this is that the orthogonal projection from $L^2(M)$ onto $\overline{L^2(M_1)+\cdots+L^2(M_m)} = \mathbb{C}\xi_\varphi \oplus L^2(M_1)^\circ \oplus\cdots\oplus L^2(M_m)^\circ$ is an extension of $x\xi_\varphi \in M\xi_\varphi \mapsto \Phi(x)\xi_\varphi = \varphi(x)\xi_\varphi + \sum_{i=1}^m E_{M_i}(x-\varphi(x)1)\xi_\varphi \in M\xi_\varphi$, where we write $L^2(M_i)^\circ := L^2(M_i) \ominus \mathbb{C}\xi_{\varphi_i}$, $1 \leq i \leq m$.

\begin{lemma}\label{L6} 
For any $x \in (M^\omega)^{(\alpha^\omega,G)}$ and $y,z \in \mathrm{Ker}(\Phi)$, the vectors
$$
y(x-\Phi^\omega(x))\xi_{\varphi^\omega}, \quad 
(y\Phi^\omega(x) - \Phi^\omega(x)z)\xi_{\varphi^\omega}, \quad 
(\Phi^\omega(x)-x)z\xi_{\varphi^\omega}
$$
are mutually orthogonal in $L^2(M^\omega)$. Here, we set $(M^\omega)^{(\alpha^\omega,G)} := \bigcap_{g \in G} (M^\omega)^{\alpha^\omega_g}$, being a von Neumann subalgebra of $M^\omega$. 
\end{lemma}

Note that we do not assume that the action $\alpha^\omega : G \curvearrowright M^\omega$ is continuous in the $u$-topology. 

\begin{proof} Thanks to the Kaplansky density theorem, one can choose bounded nets $y_\lambda$ and $z_\lambda$ of elements in the unital $*$-subalgebra (algebraically) generated by $\sigma^{\varphi_i}$-analytic elements in $M_i$, $1 \leq i \leq m$, such that $\lim_\lambda y_\lambda = y$ and $\lim_\lambda z_\lambda = z$ $\sigma$-strongly. Then the $y_\lambda - \Phi(y_\lambda)$ and the $z_\lambda - \Phi(z_\lambda)$ fall into $\mathrm{Ker}(\Phi)$ by Lemma \ref{L5}(1) and are $\sigma^\varphi$-analytic by Lemma \ref{L5}(2). Moreover, $\lim_\lambda (y_\lambda - \Phi(y_\lambda)) = y - \Phi(y) = y$ and $\lim_\lambda (z_\lambda - \Phi(z_\lambda)) = z - \Phi(z) = z$ $\sigma$-strongly, and furthermore we observe that the $y_\lambda - \Phi(y_\lambda)$ and the $z_\lambda - \Phi(z_\lambda)$ fall into the linear span of the reduced words of length $\geq 2$ whose letters are all $\sigma^{\varphi_i}$-analytic. Hence, we may assume that $y = \sum_{k=1}^\ell y_k$ and $z = \sum_{k=1}^{\ell'} z_k$ such that all $y_k$ and $z_k$ are reduced words in the $M_i^\circ := \mathrm{Ker}(\varphi_i)$ of length $\geq 2$ whose letters are all $\sigma^{\varphi_i}$-analytic. By definition we can also write $x = (x_n)^\omega$ with a bounded sequence $(x_n)$ of elements in $M$. Lemma \ref{L5}(3) shows that $x - \Phi^\omega(x) = (x_n - \Phi(x_n))^\omega$ holds for every $n$.  

\medskip
For each $1 \leq i \leq m$ we define the finite dimensional subspace $V_i$ of $M_i^\circ$ generated by all the $a, a^*, \sigma_\mathrm{i}^{\varphi_i}(a)^*$ of $M_i^\circ$-letters $a$ appearing in the $y_k$ and the $z_k$, and denote by $W_i$ the orthogonal complement of $V_i$ in $M_i^\circ$, that is, the range of the mapping $x \in M_i^\circ \mapsto x - \sum_{j=1}^{d_i} \varphi_i(v_{ij}^* x)v_{ij}\in M_i^\circ$ with an orthogonal basis $\{v_{ij}\}_{j=1}^{d_i}$ of $V_i$ with respect to the inner product $(a|b)_{\varphi_i} := \varphi_i(b^* a)$, $a, b \in M_i$. Observe that $M_i^\circ = V_i + W_i$ for $1 \leq i \leq m$. 

Consider the following subspaces $\mathcal{L}_i$, $\mathcal{R}_i$, $1 \leq i \leq m$, and $\mathcal{X}$ of $L^2(M)$ as follows. 
\begin{align*} 
\mathcal{L}_i :&= \big[\{\text{reduced words starting in $V_i$ of length $\geq 2$}\}\xi_\varphi\big], \\ 
\mathcal{R}_i :&= \big[\{\text{reduced words ending in $V_i$ of length $\geq 2$}\}\xi_\varphi\big], \\
\mathcal{X} :&= \big[\{\text{reduced words starting and ending in $W_i$ ($1 \leq i \leq m$) of length $\geq 2$}\}\xi_\varphi\big], 
\end{align*}
where $[\cdots]$ means the operation of closed linear span. Observe that 
$$
L^2(M) = \mathbb{C}\xi_\varphi \oplus L^2(M_1)^\circ\oplus\cdots\oplus L^2(M_m)^\circ \oplus \overline{(\mathcal{L}_1+\cdots+\mathcal{L}_m + \mathcal{R}_1+\cdots\mathcal{R}_m)} \oplus \mathcal{X}
$$ 
and that the mapping $a\xi_\varphi \mapsto \Phi(a)\xi_\varphi$ induces the orthogonal projection from $L^2(M)$ onto $ \mathbb{C}\xi_\varphi \oplus L^2(M_1)^\circ\oplus\cdots\oplus L^2(M_2)^\circ$. 

\medskip
The first step is to show that $x\xi_{\varphi^\omega} = (P_\mathcal{X}x_n\xi_\varphi + \Phi(x_n)\xi_\varphi)^\omega$ holds inside $L^2(M)^\omega$. To this end it suffices to prove that $\lim_{n\to\omega}\Vert P_{\mathcal{L}_i}x_n\xi_\varphi\Vert_{L^2(M)} = \lim_{n\to\omega}\Vert P_{\mathcal{R}_i}x_n\xi_\varphi\Vert_{L^2(M)} = 0$ for every $1 \leq i \leq m$. 

Set 
$$
\mathcal{H}_L(i) := \sideset{}{^\oplus}\sum_{i \neq i(1)} L^2(M_{i(1)})^\circ \oplus \sideset{}{^\oplus}\sum_{i \neq i(1) \neq i(2)} L^2(M_{i(1)})^\circ\botimes L^2(M_{i(2)})^\circ\,\oplus \cdots 
$$
(the direct sum of all the `traveling' tensor products $L^2(M_{i(1)})^\circ\,\botimes\,\cdots\,\botimes\,L^2(M_{i(\ell)})^\circ$ ($i(1) \neq i(2) \neq \cdots$) not starting at $L^2(M_i)^\circ$). Then $\mathcal{L}_i$ may be identified with $(V_i\xi_{\varphi_i})\,\botimes\,\mathcal{H}_L(i) \subseteq L^2(M_i)^\circ\,\botimes\,\mathcal{H}_L(i)$. Let $\pi_i : G \curvearrowright L^2(M_i)$ be the canonical implementing unitary representation of the action $\alpha^{(i)} : G \curvearrowright M_i$. Similarly, let $\pi : G \curvearrowright L^2(M)$ be the canonical implementing unitary representation of the action $\alpha : G \curvearrowright M$. 
 Then one has $\pi(g) = \bigstar_{i=1}^m \pi_i(g)$ for $g \in G$, that is, its restriction to the reduced subspace $L^2(M_{i(1)})^\circ\,\botimes\,\cdots\,\botimes\,L^2(M_{i(\ell)})^\circ$ is given by $\pi_{i(1)}(g)\,\otimes\,\cdots\,\otimes\,\pi_{i(\ell)}(g)$, $g \in G$. Hence the restriction of $\pi(g)$ to $\mathcal{L}_i$ is the restriction of $\pi_i(g)\,\otimes \pi_{L,i}(g)$ on $L^2(M_i)^\circ\,\botimes\,\mathcal{H}_L(i)$ to $(V_i\xi_{\varphi_i})\,\botimes\,\mathcal{H}_L(i)$ with a certain unitary representation $\pi_{L,i} : G \curvearrowright \mathcal{H}_L(i)$.  

For every $n \in \mathbb{N}$ and every $g \in G$ we have 
\begin{equation}\label{Eq2}
\begin{aligned}
\Vert P_{\mathcal{L}_i} x_n\xi_\varphi\Vert_{L^2(M)}^2 
&= 
\Vert \pi(g)P_{\mathcal{L}_i} x_n\xi_\varphi\Vert_{L^2(M)}^2 \\
&= 
\Vert \pi(g)P_{\mathcal{L}_i} x_n\xi_\varphi - P_{\pi(g)\mathcal{L}_i} x_n\xi_\varphi + P_{\pi(g)\mathcal{L}_i} x_n\xi_\varphi \Vert_{L^2(M)}^2 \\
&\leq 
2\Vert P_{\pi(g)\mathcal{L}_i} (\alpha_g(x_n) - x_n)\xi_\varphi\Vert_{L^2(M)}^2 + 2\Vert P_{\pi(g)\mathcal{L}_i} x_n\xi_\varphi \Vert_{L^2(M)}^2 \\
&\leq 
2\Vert (\alpha_g(x_n) - x_n)\xi_\varphi\Vert_{L^2(M)}^2 + 2\Vert P_{\pi(g)\mathcal{L}_i} x_n\xi_\varphi \Vert_{L^2(M)}^2. 
\end{aligned}
\end{equation} 

Let $K \in \mathbb{N}$ be arbitrarily chosen. According to \cite[Proposition 2.3]{Houdayer:TAMS14} we choose $0 < \varepsilon < 1/2$ in such a way that $\prod_{k=1}^{K-1} (1+\delta^{\circ k}(\varepsilon))^2 \leq 2$, where $\delta^{\circ k}$ means the $k$-times composition of the function 
$$
\delta : t \in [0,1/2) \mapsto \frac{2t}{\sqrt{1-t -\sqrt{2}t\sqrt{1-t}}} \in [0,+\infty). 
$$
Let $\{v_{ij}\}_{j=1}^{d_i}$ be an orthonormal basis of $V_i$ as before. Then $\{v_{ij}\xi_{\varphi_i}\}_{j=1}^{d_i}$ becomes an orthonormal basis of the (closed) subspace $V_i\xi_{\varphi_i}$ of $L^2(M_i)^\circ$. By using the weak mixing property of $\alpha^{(i)} : G \curvearrowright M_i$ one can inductively find a distinct $2^K$-tuple $g_1,\dots,g_{2^K} \in G$ in such a way that 
\begin{equation}\label{Eq3}
\max_{1 \leq j_1, j_2 \leq d_i} \big|(\pi_i(g_{k_1})v_{ij_1}\xi_{\varphi_i}\,|\,\pi_i(g_{k_2})v_{ij_2}\xi_{\varphi_i})_{L^2(M_i)}\big| \leq \frac{\varepsilon}{d_i}
\end{equation}
as long as $k_1 \neq k_2$. In fact, $g_1$ can arbitrarily be chosen, and assume that we have chosen $g_1,\dots,g_p$ in such a way that the above inequality holds for any pair $k_1 \neq k_2 \leq p$. Then, applying the weak mixing property to the finite set $\{\pi_i(g_k)v_{ij} \mid 1 \leq j \leq d_i, 1 \leq k \leq p \}$ we find $g_{p+1} \in G$ so that
$$
\big|(\pi_i(g_{p+1})v_{ij_1}\xi_{\varphi_i}\,|\,\pi_i(g_k)v_{ij_2}\xi_{\varphi_i})_{L^2(M_i)}\big| \leq \frac{\varepsilon}{d_i}
$$
for all $1 \leq j_1, j_2 \leq d_i$ and $1 \leq k \leq p$. If  $g_{p+1} = g_k$ for some $1 \leq k \leq p$, then $1 = (v_{ij}\,|\,v_{ij})_{\varphi_i} = \big|(\pi_i(g_{p+1})v_{ij}\xi_{\varphi_i}\,|\,\pi_i(g_k)v_{ij}\xi_{\varphi_i})_{L^2(M_i)}\big| \leq \varepsilon/d_i \leq 1/2$, a contradiction. Hence $g_{p+1}$ is different from $g_1,\dots,g_p$. In this way, we can inductively obtain the desired family $g_1,\dots,g_{2^K} \in G$. 

Let $\xi, \eta \in \mathcal{L}_i = (V_i\xi_{\varphi_i})\,\botimes\,\mathcal{H}_L(i)$ be arbitrarily chosen. Then we can write $\xi = \sum_{j=1}^{d_i} v_{ij}\xi_{\varphi_i}\otimes  \xi_j$ and $\eta = \sum_{j=1}^{d_i} v_{ij}\xi_{\varphi_i}\otimes  \eta_j$ with $\xi_j,\eta_j \in \mathcal{H}_L(i)$. If $k_1 \neq k_2$, then  
\begin{align*} 
&\big|(\pi(g_{k_1})\xi\,|\,\pi(g_{k_2})\eta)_{L^2(M)}\big| \\
&\leq 
\sum_{j_1,j_2 = 1}^{d_i} \big|(\pi_i(g_{k_1})v_{ij_1}\xi_{\varphi_i}\,|\,\pi_i(g_{k_2})v_{ij_2}\xi_{\varphi_i})_{L^2(M_i)}\big|\cdot \big|(\pi_{L,i}(g_{k_1})\xi_{j_1}\,|\,\pi_{L,i}(g_{k_2})\eta_{j_2})_{\mathcal{H}_L(i)}\big| \\
&\leq \frac{\varepsilon}{d_i} \sum_{j_1,j_2 = 1}^{d_i} \Vert \xi_{j_1}\Vert_{\mathcal{H}_L(i)}\cdot\Vert \eta_{j_2}\Vert_{\mathcal{H}_L(i)} \quad (\text{use the above \eqref{Eq3}}) \\
&\leq 
\varepsilon \sqrt{\sum_{j=1}^{d_i} \Vert \xi_j \Vert_{\mathcal{H}_L(i)}^2}\cdot\sqrt{\sum_{j=1}^{d_i} \Vert \eta_j \Vert_{\mathcal{H}_L(i)}^2} \quad (\text{by the Cauchy--Schwarz inequality ($d_i$-times)}) \\
&= 
\varepsilon\,\Vert \xi\Vert_{L^2(M)}\cdot\Vert \eta\Vert_{L^2(M)}. 
\end{align*}  
This means that the subspaces $\pi(g_k)\mathcal{L}_i$. $1 \leq k \leq 2^K$, are mutually $\varepsilon$-orthogonal in the sense of \cite[Definition 2.1]{Houdayer:TAMS14}. Therefore, by \cite[Proposition 2.3]{Houdayer:TAMS14} we get  
$$
\sum_{k=1}^{2^K} \Vert P_{\pi(g_k)\mathcal{L}_i} x_n\xi_\varphi \Vert_{L^2(M)}^2 \leq 4 \Vert  x_n\xi_\varphi \Vert_{L^2(M)}^2,
$$
and this and inequality \eqref{Eq2} imply that 
$$
2^K \Vert P_{\mathcal{L}_i} x_n\xi_\varphi\Vert_{L^2(M)}^2 \leq 2\sum_{k=1}^{2^K}\Vert (\alpha_{g_k}(x_n) - x_n)\xi_\varphi\Vert_{L^2(M)}^2 + 4 \Vert  x_n\xi_\varphi \Vert_{L^2(M)}^2. 
$$
Since $\alpha_g^\omega(x) = x$ for all $g \in G$, we obtain that $\lim_{n\to\omega} \Vert P_{\mathcal{L}_i} x_n\xi_\varphi\Vert_{L^2(M)}^2 \leq 2^{2-K} \Vert  x\xi_{\varphi^\omega} \Vert_{L^2(M^\omega)}^2$. Since $K$ can arbitrarily be large, we conclude that $\lim_{n\to\omega} \Vert P_{\mathcal{L}_i} x_n\xi_\varphi\Vert_{L^2(M)} = 0$. 

In the same way, we also obtain that $\lim_{n\to\omega} \Vert P_{\mathcal{R}_i} x_n\xi_\varphi\Vert_{L^2(M)} = 0$. Therefore, we have finished the first step. 

\medskip
The second step is to show that $y\mathcal{X}$, $y(M_1+\cdots+M_m)\xi_\varphi + (M_1+\cdots+M_m)z\xi_\varphi$ and $J^M\sigma_{-\mathrm{i}/2}^\varphi(z^*)J^M\mathcal{X}$ are mutually orthogonal in $L^2(M)$. Since the adjoint $a^*$ and $\sigma_\mathrm{i}^{\varphi_i}(a)^*$ of any $M_i^\circ$-letter $a$ in the $y_k$ and $z_k$ fall in $V_i$, it is plain to see that $y\mathcal{X}$ and $J^M\sigma_{-\mathrm{i}/2}^\varphi(z^*)J^M\mathcal{X}$ sit in 
\begin{align*} 
&\big[\{\text{reduced words starting in some $V_i$ and ending in some $W_i$ of length $\geq 3$}\}\xi_\varphi\big], \\
&
\big[\{\text{reduced words starting in some $W_i$ and ending in some $V_i$ of length $\geq 3$}\}\xi_\varphi\big],
\end{align*}
respectively. The choice of the subspaces $V_i$, $W_i$ guarantees that $y\mathcal{X}$ and $J^M\sigma_{-\mathrm{i}/2}^\varphi(z^*)J^M\mathcal{X}$ are orthogonal. Since $y_k^* y_{k'}$ is a linear combination of $1$ and reduced words in the $V_i$, any element in $y_k^* y_{k'} M_i$ is a linear combination of an element in $M_i^\circ$ and reduced words starting in $V_i$, and hence $\mathcal{X}$ and $y_k^*y_{k'} M_i\xi_\varphi$ are orthogonal, so that so are $y_k \mathcal{X}$ and $y_{k'}M_i\xi_\varphi$. Since any element of $M_i z_k$ is a linear combination of an element in $M_i^\circ$ and reduced words ending in some $V_i$, $y\mathcal{X}$ and $M_i z_k\xi_\varphi$ are orthogonal. Therefore, we have confirmed that $y\mathcal{X}$ and $y(M_1+\cdots+M_m)\xi_\varphi + (M_1+\cdots+M_m)z\xi_\varphi$ are orthogonal. In the same way, we can confirm that $y(M_1+\cdots+M_m)\xi_\varphi + (M_1+\cdots+M_m)z\xi_\varphi$ and $J^M\sigma_{-\mathrm{i}/2}^\varphi(z^*)J^M\mathcal{X}$ are orthogonal. Hence we have finished the second step.  

\medskip
Let us finalize this proof. By the first step we have
\begin{align*}
y(x-\Phi^\omega(x))\xi_{\varphi^\omega} 
&= 
(y(x_n - \Phi(x_n))\xi_\varphi)^\omega = (yP_\mathcal{X} x_n\xi_\varphi)^\omega, \\ 
(y\Phi^\omega(x)-\Phi^\omega(x)z)\xi_{\varphi^\omega} 
&= 
((y\Phi(x_n) - \Phi(x_n)z)\xi_\varphi)^\omega, \\
(\Phi^\omega(x)-x)z\xi_{\varphi^\omega} 
&= 
(-J^M\sigma_{-\mathrm{i}/2}^\varphi(z^*)J^M(x_n-\Phi(x_n))\xi_\varphi)^\omega \\ 
&= (-J^M\sigma_{-\mathrm{i}/2}^\varphi(z^*)J^M P_\mathcal{X} x_n\xi_\varphi)^\omega 
\end{align*} 
in $L^2(M)^\omega$. The second step shows that for every $n \in \mathbb{N}$ the vectors $yP_\mathcal{X} x_n\xi_\varphi$, $(y\Phi(x_n) - \Phi(x_n)z)\xi_\varphi$ and $J^M\sigma_{-\mathrm{i}/2}^\varphi(z^*)J^M P_\mathcal{X} x_n\xi_\varphi$ are mutually orthogonal. Thus, the desired assertion immediately follows.  
\end{proof}

The next lemma is a kind of rigidity result on free products of amenable von Neumann algebras and will work for showing the rigidity of the number of free components in the proof of Theorem \ref{T1}. 

\begin{lemma}\label{L7} Any amenable von Neumann subalgebra $Q$ of $M$ with separable predual such that $Q'\cap(M^\omega)^{(\alpha^\omega,G)}$ is diffuse and with expectation must embed into $M_i$ inside $M$, that is, $Q \preceq_M M_i$, for some $1 \leq i \leq m$. 
\end{lemma}
\begin{proof} 
Suppose on the contrary that $Q \not\preceq_M M_i$ for all $1 \leq i \leq m$. Proposition \ref{P3} together with a direct sum trick (see the proof of \cite[Theorem 4.3]{IoanaPetersonPopa:ActaMath08} due to Ioana--Peterson--Popa) enables us to find a single net $(u_j)_{j\in J}$ of unitaries in $Q$ in such a way that $\lim_j \varphi(u_j) = 0$ as well as $\lim_j E_{M_i}(u_j) = 0$ strongly for every $1 \leq i \leq m$. In particular, $\lim_j \Phi(u_j) = 0$ strongly. Let $x \in Q' \cap (M^\omega)^{(\alpha^\omega,G)}$ be arbitrarily chosen. By Lemma \ref{L6}, the vector $[u_j-\Phi(u_j), x]\xi_{\varphi^\omega}$ is decomposed into a sum of orthogonal ones $(u_j-\Phi(u_j))(x-\Phi^\omega(x))\xi_{\varphi^\omega}$, $[(u_j-\Phi(u_j)),\Phi^\omega(x)]\xi_{\varphi^\omega}$ and $(\Phi^\omega(x)-x)(u_j-\Phi(u_j))\xi_{\varphi^\omega}$ in $L^2(M^\omega)$, and hence
\begin{align*}
\big\Vert (1 - u_j^* \Phi(u_j))(x - \Phi^\omega(x))\xi_{\varphi^\omega}\big\Vert_{L^2(M^\omega)} 
&= 
\big\Vert (u_j - \Phi(u_j))(x - \Phi^\omega(x))\xi_{\varphi^\omega}\big\Vert_{L^2(M^\omega)} \\
&\leq 
\big\Vert [u_j - \Phi(u_j),x]\xi_{\varphi^\omega} \big\Vert_{L^2(M^\omega)} \\
&= \big\Vert [x,\Phi(u_j)]\xi_{\varphi^\omega} \big\Vert_{L^2(M^\omega)}. 
\end{align*}
Taking the limit of this inequality along $j\in J$ we conclude that $x = \Phi^\omega(x) \in M_1^\omega + \cdots M_m^\omega$. Hence we get $P := Q'\cap(M^\omega)^{(\alpha^\omega,G)} \subseteq M_1^\omega + \cdots + M_m^\omega$.

\medskip
We claim the following: $P \subseteq M_{i_0}^\omega$ for some $1 \leq i_0 \leq m$. 

\medskip
Let $x \in P$ be arbitrarily chosen. Since $P \subseteq M_1^\omega + \cdots + M_m^\omega = \mathbb{C}1 + (M_1^\omega)^\circ + \cdots + (M_m^\omega)^\circ$, we write $x = \alpha 1 + \sum_{i=1}^m x_i^\circ$ with $x_i^\circ \in (M_i^\omega)^\circ := \mathrm{Ker}(\varphi_i^\omega)$. This decomposition is uniquely determined thanks to the free independence among the $M_i^\omega$ (see \cite[Proposition 4]{Ueda:TAMS03}). Then we have 
$$
\underbrace{(\alpha^2 + \sum_{i=1}^m \varphi^\omega((x_i^\circ)^2))1}_{\in \mathbb{C}1} + \underbrace{\sum_{i=1}^m (2\alpha x_i^\circ + ((x_i^\circ)^2 - \varphi^\omega((x_i^\circ)^2)1))}_{\in (M_1^\omega)^\circ + \cdots + (M_m^\omega)^\circ} + \sum_{1 \leq i\neq j \leq m} \underbrace{x_i^\circ x_j^\circ}_{\in (M_i^\omega)^\circ (M_j^\omega)^\circ} = x^2
$$
falls into $P \subseteq \mathbb{C}1 + (M_1^\omega)^\circ + \cdots + (M_m^\omega)^\circ$. Since the $M_i^\omega$ are freely independent, we observe that if $i \neq j$, then $x_i^\circ x_j^\circ$ must be $0$. Hence there exists no pair $i \neq j$ such that both $x_i^\circ \neq 0$ and $x_j^\circ \neq 0$ hold. It follows that each $x \in P$ falls into $M_i^\omega$ for some $1 \leq i \leq m$. 

Choose $x \in P \setminus \mathbb{C}1$. (Note that $P$ is diffuse by assumption.) As shown above there exists $1\leq i_0 \leq m$ so that $x \in M_{i_0}^\omega$. Suppose that there exist $1 \leq j \leq m$ with $j\neq i_0$ and $y \in P$ such that $y \in M_j^\omega \setminus M_{i_0}^\omega = M_j^\omega \setminus \mathbb{C}1$. Then we can write $x = \alpha 1 + x^\circ$ and $y = \beta 1 + y^\circ$ with $x^\circ \in (M_{i_0}^\omega)^\circ$ and $y^\circ \in (M_j^\omega)^\circ$. By the choice of $x$ and $y$, we observe that $x^\circ \neq 0$ and $y^\circ \neq 0$. But, as above, we have $\alpha\beta1 + (\alpha y^\circ + \beta x^\circ) + x^\circ y^\circ \in M_1^\omega + \cdots + M_m^\omega$, implying that either $x^\circ=0$ or $y^\circ=0$, a contradiction. Therefore, all the other $y \in P$ fall into the $M_{i_0}^\omega$ containing the element $x$ that we initially chose. We have proved the claim. 

\medskip
Consider the von Neumann subalgebra $\bigvee_{i=1}^m M_i^\omega$ of $M^\omega$ generated by the $M_i^\omega$, and observe that $(\bigvee_{i=1}^m M_i^\omega,\varphi^\omega\!\upharpoonright_{\bigvee_{i=1}^m M_i^\omega})$ is nothing but the free product of the $(M_i^\omega,\varphi_i^\omega)$. Thus, working inside $M^\omega$ we have 
$$
Q \subset P' \cap M = P' \cap \big(\bigvee_{i=1}^m M_i\big) \subseteq P' \cap \big(\bigvee_{i=1}^m M_i^\omega\big) \subseteq M_{i_0}^\omega
$$ 
by \cite[Proposition 2.7(i)]{HoudayerUeda:MPCPS16}  (see also \cite[Proposition 3.1]{Ueda:AdvMath11}), since $P \subseteq M_{i_0}^\omega$ is diffuse and with expectation by assumption. Since the restriction of $E_{M_{i_0}^\omega}$ to $M$ is $E_{M_{i_0}}$ by definition, we have 
$$
Q \subseteq M_{i_0}^\omega \cap M = E_{M_{i_0}^\omega}(M_{i_0}^\omega \cap M) \subseteq E_{M_{i_0}^\omega}(M) = E_{M_{i_0}}(M) = M_{i_0} 
$$
in contradiction to $Q \not\preceq_M M_{i_0}$.     
\end{proof} 

One may think that the same statement as Lemma \ref{L7} holds even when $Q$ is not amenable but has sufficiently many central sequences (i.e, non-full). However, this case was already treated as case (ii) in the proof of \cite[Main Theorem]{HoudayerUeda:ComposMath16} without using any group actions. Its proof is rather different from the present one. 

\medskip
The next lemma is a key observation for the part of `state-rigidity' in Theorem \ref{T1}.  

\begin{lemma}\label{L8} Let $u \in M$ be a unitary and $1 \leq i \leq m$ be arbitrarily given. Then, if $M_i$ is diffuse and if $uM_i u^*$ is globally invariant under the action $\alpha$, then $u$ must fall into $M_i$, and hence $uM_i u^* = M_i$.
\end{lemma}  
\begin{proof} We use the same notation as in the proof of Lemma \ref{L6}. 

By assumption $\alpha_g(u)M_i\alpha_g(u)^* = uM_i u^*$ and hence $M_i = (\alpha_g(u)^*u) M_i (\alpha_g(u)^*u)^*$ for all $g \in G$.  By \cite[Proposition 2.7(i)]{HoudayerUeda:MPCPS16} we have $v_g := (\alpha_g(u)^*u) \in M_i$ for every $g \in G$. Observe that $\alpha_g\circ E_{M_i} = E_{M_i}\circ\alpha_g$ holds for every $g \in G$. Letting $x := u^* - E_{M_i}(u^*) \in \mathrm{Ker}(E_{M_i})$ we obtain that $\alpha_g(x) = v_g x$ for all $g \in G$. It suffices to prove that $x=0$. 

\medskip
Observe that $L^2(M)\ominus L^2(M_i)$ is decomposed into the direct sum of subspaces of the form 
$$
\mathfrak{X} := \big[M_i M_{i(1)}^\circ \cdots M_{i(\ell)}^\circ M_i \xi_\varphi\big] = L^2(M_i)\botimes L^2(M_{i(1)})^\circ\botimes\cdots\botimes L^2(M_{i(\ell)})^\circ\botimes L^2(M_i)
$$
with $\ell \geq 1$ and $i \neq i(1) \neq i(2) \neq \cdots \neq i(\ell) \neq i$, and thus it suffices to prove that $P_\mathfrak{X} x\xi_\varphi = 0$ for any such $\mathfrak{X}$, where $P_\mathfrak{X}$ denotes the projection onto $\mathfrak{X}$ in $L^2(M)$. We remark that $P_\mathfrak{X}$ commutes with the $v_g$ and  the $\pi(g)$.

\medskip
For an arbitrary $\varepsilon > 0$, we can find a sum $y = \sum_{k=1}^K w_k$ of words $w_k$ in $M_i M_{i(1)}^\circ \cdots M_{i(\ell)}^\circ M_i$ such that $(2\Vert x \Vert + \Vert y\xi_\varphi - P_\mathfrak{X}x\xi_\varphi\Vert_{L^2(M)}) \Vert y\xi_\varphi - P_\mathfrak{X}x\xi_\varphi\Vert_{L^2(M)} \leq \varepsilon/2$. Since 
$$
\pi(g)P_\mathfrak{X}x\xi_\varphi = P_\mathfrak{X}\pi(g)x\xi_\varphi = P_\mathfrak{X}\alpha_g(x)\xi_\varphi = P_\mathfrak{X}v_g x\xi_\varphi = v_g P_\mathfrak{X}x\xi_\varphi, 
$$
we have 
\begin{align*} 
&\big| (\pi(g)y\xi_\varphi |v_g y\xi_\varphi)_{L^2(M)} - \Vert P_\mathfrak{X} x\xi_\varphi\Vert^2_{L^2(M)}\big| \\ 
&=
\big| (\pi(g)y\xi_\varphi |v_g y\xi_\varphi)_{L^2(M)} -  (\pi(g)P_\mathfrak{X}x\xi_\varphi |v_g P_\mathfrak{X}x\xi_\varphi)_{L^2(M)}\big| \\
&\leq 
\big| (\pi(g)(y\xi_\varphi - P_\mathfrak{X}x)\xi_\varphi|v_g y\xi_\varphi)_{L^2(M)} \big| + \big|(\pi(g)P_\mathfrak{X}x\xi_\varphi | v_g(y\xi_\varphi-P_\mathfrak{X}x\xi_\varphi))_{L^2(M)} \big| \\
&\leq 
\Vert y\xi_\varphi\Vert_{L^2(M)}\cdot\Vert y\xi_\varphi - P_\mathfrak{X}x\xi_\varphi\Vert_{L^2(M)} + \Vert x\Vert\cdot\Vert y\xi_\varphi - P_\mathfrak{X}x\xi_\varphi\Vert_{L^2(M)} \\
&\leq 
\big(\Vert P_\mathfrak{X}x\xi_\varphi\Vert_{L^2(M)} + \Vert y\xi_\varphi - P_\mathfrak{X}x\xi_\varphi\Vert_{L^2(M)}\big)\Vert y\xi_\varphi - P_\mathfrak{X}x\xi_\varphi\Vert_{L^2(M)} \\
&\qquad\qquad\qquad\qquad\qquad\qquad\qquad\qquad+ 
\Vert x\Vert\cdot\Vert y\xi_\varphi - P_\mathfrak{X}x\xi_\varphi\Vert_{L^2(M)} \\
&\leq 
(2\Vert x \Vert + \Vert y\xi_\varphi - P_\mathfrak{X}x\xi_\varphi\Vert_{L^2(M)}) \Vert y\xi_\varphi - P_\mathfrak{X}x\xi_\varphi\Vert_{L^2(M)}, 
\end{align*}
and hence 
\begin{align*}
\Vert P_\mathfrak{X} x\xi_\varphi\Vert_{L^2(M)}^2 
&\leq 
\big| (\pi(g)y\xi_\varphi |v_g y\xi_\varphi)_{L^2(M)}\big| + \frac{\varepsilon}{2} \\
&\leq 
\sum_{k_1,k_2=1}^K \big| (\pi(g) w_{k_1}\xi_\varphi |v_g  w_{k_2}\xi_\varphi)_{L^2(M)}\big| + \frac{\varepsilon}{2}
\end{align*}
for every $g \in G$. For each $1 \leq k \leq K$, we write $w_k = a_k w'_k$, where $a_k \in M_i$ is the first $M_i$-letter in $w_k$ and $w'_k$ denotes the remaining word obtained by removing the first $M_i$-letter from $w_k$. Observe that $\mathfrak{X}$ is naturally identified with $L^2(M_i)\botimes \mathfrak{Y}$, where 
\begin{equation}\label{Eq4}
\mathfrak{Y} := \big[M_{i(1)}^\circ \cdots M_{i(\ell)}^\circ M_i \xi_\varphi\big] = L^2(M_{i(1)})^\circ\botimes\cdots\botimes L^2(M_{i(\ell)})^\circ\botimes L^2(M_i).
\end{equation}
This identification intertwines the restriction of $\pi$ to $\mathfrak{X}$ with the tensor product representation of $\pi_i$ and the restriction of $\pi$ to $\mathfrak{Y}$, and sends each $w_k\xi_\varphi$ to $a_k\xi_{\varphi_i}\otimes w'_k\xi_\varphi$. Observe that 
\begin{align*}
&\big|(\pi(g)w_{k_1}\xi_\varphi|v_g w_{k_2}\xi_\varphi)_{L^2(M)}\big| \\
&= 
\big|(\pi_i(g)a_{k_1}\xi_{\varphi_i}|v_g a_{k_2}\xi_{\varphi_i})_{L^2(M_i)}\big|\cdot\big|(\pi(g)w'_{k_1}\xi_\varphi|w'_{k_2}\xi_\varphi)_{L^2(M)}\big| \\
&\leq 
\big(\max_{1 \leq k \leq K} \Vert a_k \Vert\big)^2\,\big|(\pi(g)w'_{k_1}\xi_\varphi|w'_{k_2}\xi_\varphi)_{L^2(M)}\big|  
\end{align*}
for all $1 \leq k_1, k_2 \leq K$. The restriction of $\pi(g)$ to $\mathfrak{Y}$ is identified with $\pi_{i(1)}(g)\otimes\cdots\otimes\pi_{i(\ell)}(g)\otimes \pi_i(g)$ via the tensor product decomposition in \eqref{Eq4}, and thus $\pi : G \curvearrowright \mathfrak{Y}$ is weakly mixing. Hence there exists $g_\varepsilon \in G$ so that  
$$
\sum_{k_1,k_2=1}^K \big|(\pi(g_\varepsilon)w_{k_1}\xi_\varphi|v_{g_\varepsilon} w_{k_2}\xi_\varphi)_{L^2(M)}\big| \leq \frac{\varepsilon}{2}. 
$$
Consequently, we get $\Vert P_\mathfrak{X} x\xi_\varphi\Vert_{L^2(M)}^2 \leq \varepsilon$. Since $\varepsilon>0$ can arbitrarily be small, we conclude that $P_\mathfrak{X} x\xi_\varphi = 0$. Hence we are done.  
\end{proof}  

We are ready to prove Theorem \ref{T1}.

\medskip\noindent
{\it Proof of Theorem \ref{T1}.} For simplicity, we write $M_i := \lambda^M_i(R_\infty)$ ($1\leq i \leq m$) and $N_j := \lambda^N_j(R_\infty)$ ($1 \leq j \leq n$). In what follows we write $[m] := \{1,2,\dots,m\}$ and similarly $[n] := \{1,2,\dots,n\}$. 

In order to prove the theorem, we assume that there exists a bijective $*$-homomorphism $\pi : M \to N$ such that $\psi = \pi_*(\varphi)$. Hence we have $\pi\circ\sigma_t^\varphi = \sigma_t^\psi\circ\pi$, $t \in \mathbb{R}$, by confirming the so-called modular condition. Observe that $N_j$ is globally invariant under $\sigma^\psi$ and hence so is $\pi^{-1}(N_j)$ under $\sigma^\varphi$. Thanks to \cite[Theorem 4.1]{AndoHaagerup:JFA14} we have 
$$
\pi^{-1}(N_j)' \cap (M^\omega)^{(\sigma^{\varphi^\omega}, \mathbb{R})} =  (\pi^{-1}(N_j)'\cap M^\omega)^{(\sigma^{\varphi^\omega}, \mathbb{R})} \supseteq (\pi^{-1}(N_j)'\cap \pi^{-1}(N_j)^\omega)^{(\sigma^{\varphi^\omega}, \mathbb{R})}, 
$$
and the rightmost algebra is isomorphic to the asymptotic centralizer of $R_\infty$ by \cite[Proposition 4.35]{AndoHaagerup:JFA14}. Since $R_\infty$ is an amenable (or hyperfinite) type III$_1$ factor, $\pi^{-1}(N_j)' \cap (M^\omega)^{(\sigma^{\varphi^\omega}, \mathbb{R})}$ must be diffuse thanks to e.g. \cite[Proposition 1.1]{Takesaki:Dixmier60}. Moreover, it is rather trivial in this case where $\pi^{-1}(N_j)' \cap (M^\omega)^{(\sigma^{\varphi^\omega}, \mathbb{R})}$ is with expectation. Hence, Lemma \ref{L7} (with $G = \mathbb{R}$ and $\alpha^{(i)} = \sigma^{\varphi_i}$) shows that there exists a map $\kappa_N : [n] \to [m]$ such that $\pi^{-1}(N_j) \preceq_M M_{\kappa_N(j)}$ for every $j \in [n]$. Therefore, for each $j\in [n]$ there exist $d_N(j) \in \mathbb{N}$, a normal $*$-homomorphism $\Psi : \pi^{-1}(N_j) \to \mathbb{M}_{d_N(j)}(M_{\kappa_N(j)})$ and a non-zero partial isometry $v_N(j) \in \mathbb{M}_{d_N(j),1}(M)$ such that $\Psi(\pi^{-1}(N_j))$ is with expectation (see Remark \ref{R4}) and that $v_N(j)x = \Psi(x)v_N(j)$ for every $x \in \pi^{-1}(N_j)$. In particular, $v_N(j)^* v_N(j) = 1$ thanks to $\pi^{-1}(N_j)'\cap M = \mathbb{C}1$ by \cite[Proposition 2.7(i)]{HoudayerUeda:MPCPS16}. Remark that the normalizer of $\pi^{-1}(N_j)$ in $M$ generates $\pi^{-1}(N_j)$ itself due to \cite[Proposition 2.7(i)]{HoudayerUeda:MPCPS16}, and thus \cite[Proposition 2.7(ii)]{HoudayerUeda:MPCPS16} shows that $v_N(j)v_N(j)^* \in \mathbb{M}_{d_N(j)}(M_{\kappa_N(j)})$ and 
$$
v_N(j)\pi^{-1}(N_j)v_N(j)^* \subseteq v_N(j)v_N(j)^*\mathbb{M}_{d_N(j)}(M_{\kappa_N(j)})v_N(j)v_N(j)^*. 
$$
By symmetry, there also exists a map $\kappa_M : [m] \to [n]$ such that $\pi(M_i) \preceq_N N_{\kappa_M(i)}$, that is, $M_i \preceq_M \pi^{-1}(N_{\kappa_M(i)})$ for every $i \in [m]$. Moreover, for each $i \in [m]$ there exist $d_M(i) \in \mathbb{N}$, a partial isometry $v_M(i) \in \mathbb{M}_{1,d_M(i)}(M)$ with $v_M(i)^* v_M(i) =1$ such that $v_M(i)v_M(i)^* \in \mathbb{M}_{d_M(i)}(\pi^{-1}(N_{\kappa_M(i)}))$ and 
$$
v_M(i)M_i v_N(j)^* \subseteq v_M(i)v_M(i)^*\mathbb{M}_{d_M(i)}(\pi^{-1}(N_{\kappa_M(i)}))v_M(i)v_M(i)^*.
$$

With these facts we can proceed exactly in the same way of the final part of the proof of \cite[Main Theorem]{HoudayerUeda:ComposMath16}. However, the present situation allows us to give a bit simpler proof. To make this paper self-contained, we do give it here. Since the $M_i$ and the $\pi^{-1}(N_j)$ are all type III factors, we can make the $v_N(j)$ and the $v_M(i)$ unitaries in $M$. Thus, for each $i \in [m]$ we have 
$$
v_N(\kappa_M(i))(v_M(i)M_i v_M(i)^*)v_N(\kappa_M(i))^* 
\subseteq v_N(\kappa_M(i))\pi^{-1}(N_{\kappa_M(i)})v_N(\kappa_M(i))^*
\subseteq M_{\kappa_N(\kappa_M(i))}. 
$$
Hence, by \cite[Lemma 2.8]{HoudayerUeda:ComposMath16} we get $\kappa_N(\kappa_M(i)) = i$ for all $i \in [m]$. Then, for every $i\in[m]$ we also get  
$$
v_N(\kappa_M(i))v_M(i) M_i v_M(i)^* v_N(\kappa_M(i))^* \subseteq M_i, 
$$
which is actually equality because $v_N(\kappa_M(i))v_M(i)$ must fall into $M_i$ by \cite[Proposition 2.7(i)]{HoudayerUeda:MPCPS16}. Therefore, we have $$
v_N(\kappa_M(i))v_M(i)M_i v_M(i)^* v_N(\kappa_M(i))^* = v_N(\kappa_M(i))\pi^{-1}(N_{\kappa_M(i)})v_N(\kappa_M(i))^*,
$$
implying that 
$$
v_M(i)M_i v_M(i)^* = \pi^{-1}(N_{\kappa_M(i)})
$$
for all $i \in [m]$. By symmetry, we also obtain that $\kappa_M(\kappa_N(j)) = j$ for all $j \in [n]$. Therefore, $m=n$. Since $v_M(i)M_i v_M(i)^* = \pi^{-1}(N_{\kappa_M(i)})$ is globally invariant under $\sigma^\varphi$, Lemma \ref{L8} shows that $M_i = \pi^{-1}(N_{\kappa_M(i)})$; hence we obtain that $\pi(M_i) = N_{\kappa_M(i)}$ for all $i \in [m]$. Therefore, $\kappa := \kappa_M \in \mathfrak{S}_m$ is the desired permutation and $\pi_i := (\lambda^N_{\kappa(i)})^{-1}\circ(\pi\!\upharpoonright_{M_i})\circ\lambda^M_i$ is a well-defined $*$-automorphism of $R_\infty$. Moreover, we have 
\begin{align*}
(\pi_i)_*(\varphi_i) 
&= 
\varphi_i\circ\pi_i^{-1} \\
&= 
\varphi\circ\lambda_i^M\circ(\lambda^M_i)^{-1}\circ(\pi\!\upharpoonright_{M_i})^{-1}\circ\lambda^N_{\kappa(i)} \\
&= 
\varphi\circ(\pi^{-1}\!\upharpoonright_{N_{\kappa(i)}})\circ\lambda^N_{\kappa(i)} \\
&= 
((\varphi\circ\pi^{-1})\!\upharpoonright_{N_{\kappa(i)}})\circ\lambda^N_{\kappa(i)} \\
&= 
(\psi\!\upharpoonright_{N_{\kappa(i)}})\circ\lambda^N_{\kappa(i)} = \psi_{\kappa(i)}.
\end{align*}
Observe that $N_{j_1} = N_{j_2}$ implies that $N_{j_1}^\circ\xi_\psi = N_{j_2}^\circ\xi_\psi$ and hence $j_1 = j_2$, otherwise $N_{j_1}^\circ\xi_\psi \perp N_{j_2}^\circ\xi_\psi$ in $L^2(N)$ leading to a contradiction. It follows that the above $\kappa$ is uniquely determined by $\pi(M_i) = N_{\kappa(i)}$ for all $1 \leq i \leq m$. Hence we have proved the first (and main) part of Theorem \ref{T1}. 

\medskip
For the converse statement (the easier part), we assume that $m=n$ and there exist a bijection $\kappa : [m] \to [m]$ and $*$-automorphisms $\pi_i \in \mathrm{Aut}(R_\infty)$, $1 \leq i \leq m$ such that $(\pi_i)_*(\varphi_i) = \psi_{\kappa(i)}$ for every $1 \leq i \leq m$. Consider the new embedding maps $\rho_i^M := \lambda_{\kappa(i)}^N\circ\pi_i : R_\infty \to N$, $1 \leq i \leq m$. We observe that 
$$
\psi\circ\rho_i^M 
= 
\psi\circ\lambda_{\kappa(i)}^N\circ\pi_i
= 
\psi_{\kappa(i)}\circ\pi_i 
= 
(\pi_i)_*(\varphi_i)\circ\pi_i 
= 
(\varphi_i\circ\pi_i^{-1})\circ\pi_i 
= 
\varphi_i
$$ 
for every $1 \leq i \leq m$. Hence the characterization of free products based on the free independence guarantees that there exists a unique bijective $*$-homomorphism $\pi : M \to N$ such that $\pi\circ\lambda_i^M = \rho_i^M = \lambda_{\kappa(i)}^N\circ\pi_i$ holds for every $1 \leq i \leq m$. By the construction of $\pi$ we also have $\pi_*(\varphi) = \psi$. Hence we are done. \qed 

\section{Concluding Remarks}  

It is known that the unique amenable type II$_1$ factor $R$ admits various weakly mixing actions of second countable, locally compact groups such as the integers $\mathbb{Z}$. For each natural number $m \geq 2$ we consider the tracial free product type II$_1$ factor $M := R^{\star m}$, which is known to be isomorphic to the free group factor $L(\mathbb{F}_m)$ by Dykema \cite{Dykema:PacificJMath94} using Voiculescu's free probability theory (see e.g. \cite{VDN}). The proof of Theorem \ref{T1} (see also the proof of Proposition \ref{P10} below) actually shows the following: \emph{If any irreducible amenable type II$_1$ subfactor $Q \subset M$ with $Q' \cap M^\omega \not\subseteq Q^\omega$ had weakly mixing actions $\alpha^{(i)} : G \curvearrowright R$, $1 \leq i \leq m$, of a second countable, locally compact group $G$ such that $Q' \cap (M^\omega)^{(\gamma^\omega,G)}$ was diffuse with $\gamma := \bigstar_{i=1}^m \alpha^{(i)} : G \curvearrowright M$, then it would follow that $L(\mathbb{F}_{r_1}) \cong L(\mathbb{F}_{r_2})$ $\Longrightarrow$ $r_1=r_2$ for any integers $r_1,r_2 \geq 2$.} Note that the assumption that $Q' \cap M^\omega \not\subseteq Q^\omega$ above comes from Popa's spectral gap result \cite[Lemma 2]{Popa:IMRN07}. Remark also that this implication needs only Lemma \ref{L7}. This strategy to the non-isomorphism problem may not work well, but it seems natural (at least to us) to ask the following question:

\begin{question} Let $\gamma : G \curvearrowright M$ be as above. How large is $Q' \cap (M^\omega)^{(\gamma^\omega,G)}$ in $M^\omega$ for an (irreducible) amenable type II$_1$ subfactor $Q$ of $M$ provided that $Q'\cap M^\omega \not\subseteq Q^\omega$ ?
\end{question}

We are going to discuss this question in future. However, we can prove the next proposition by available techniques.

\begin{proposition}\label{P10}
Assume that $\alpha : G \curvearrowright R$ is a weakly mixing action of a countable, amenable group. Then, any $*$-automorphism of $M = R^{\star m}$ that commutes with the $m$-fold free product action $\alpha^{\star m} : G \curvearrowright M$ is obtained as the composition of a permutation over the free components and a free product of $*$-automorphisms on $R$ that commute with the given action $\alpha$.  
\end{proposition}
\begin{proof} Remark first that $R'\cap (R^\omega)^{(\alpha^\omega,G)} = (R_\omega)^{(\alpha^\omega,G)}$ is well known to be of type II$_1$ \cite[Lemma 8.3]{Ocneanu:LNM85}. Denote by $\beta$ such a $*$-automorphism of $M$, and also by $M_i$ the $i$th free component (isomorphic to $R$). In the same way as in the proof of Theorem \ref{T1} we can prove that there exist a permutation $\kappa \in \mathfrak{S}_m$, partial isometries $v_i \in \mathbb{M}_{d(i),1}(M)$ with $d(i) \in \mathbb{N}$ so that $v_i^* v_i = 1$ in $M_{\kappa(i)}$ (a corner of $\mathbb{M}_{d(i)}(M)$), $v_i v_i^* \in \mathbb{M}_{d(i)}(\beta^{-1}(M_i))$ and $v_i M_{\kappa(i)} v_i^* = v_i v_i^* \mathbb{M}_{d(i)}(\beta^{-1}(M_i)) v_i v_i^*$ for every $1 \leq i \leq m$. Since $M_{\kappa(i)} \cong \beta^{-1}(M_i)$ is a type II$_1$ factor, it is plain to select a unitary $u_i \in M$ in such a way that $u_i M_{\kappa(i)} u_i^* = \beta^{-1}(M_i)$ holds. Then Lemma \ref{L8} implies that $M_{\kappa(i)} = u_i M_{\kappa(i)} u_i^* = \beta^{-1}(M_i)$ holds, and the desired assertion immediately follows.       
\end{proof}

In closing, we point out that Theorem \ref{T1} as well as Proposition \ref{P10} hold under a variety of other assumptions (with allowing infinite index sets in some instances) thanks to \cite[Main Theorem]{HoudayerUeda:ComposMath16}. The details about such generalizations are left to the reader. 

\section*{Acknowledgement} 

Part of the preparation of this manuscript was done during my participation in the Hausdorff Trimester Program `Von Neumann Algebras', and we are grateful for the financial support and the hospitality of the Hausdorff Research Institute for Mathematics in Bonn. We also thank the referee for useful comments and pointing out typos. This work was supported in part by my previous Grant-in-Aid for Scientific Research (C) JP24540214 as well as my on-going Grant-in-Aid for Challenging Exploratory Research JP16K13762.

\end{document}